\documentclass[a4paper, 12pt]{article}
\usepackage{amsthm}
\usepackage{extpfeil}
\usepackage[
backend=biber,
style=numeric
]{biblatex}
\DeclareNameAlias{default}{last-first}
\usepackage{subcaption}
\usepackage{geometry}
\usepackage{lipsum}
\usepackage{tikz}
\usetikzlibrary{decorations.pathreplacing,shapes.misc}
\usepackage{pgfplots}
\usepackage{setspace}
\pgfplotsset{compat=1.11}
\usepackage{mathtools}
\usepackage{changepage}
\usepackage{theoremref}
\usepackage[utf8]{inputenc}
\newlength{\abstractwidth}
\renewcommand{\thanks}[1]{\footnote{#1}} 

\newcommand{\be}{\begin{equation}}
\newcommand{\bea}{\begin{eqnarray}}
\newcommand{\eea}{\end{eqnarray}} \newcommand{\ee}{\end{equation}}
 
 \def\ba{\begin{eqnarray}}
\def\ea{\end{eqnarray}}

\def\det{{\rm det}}

\def\log{\,{\rm log}\,}

\def\[{{\bf [}}
\def\]{{\bf ]}}

\begin{document}
\newtheorem{theorem}{Theorem}
\newtheorem{proposition}{Proposition}
\newtheorem{lemma}{Lemma}
\newtheorem{corollary}{Corollary}
\newtheorem{definition}{Definition}
\newtheorem{conjecture}{Conjecture}
\newtheorem{example}{Example}
\newtheorem{claim}{Claim}
\newtheorem{remark}{Remark}

\begin{centering}
 
\textup{\LARGE\bf Inverse Mean Curvature Flow over \\ \vspace{0.3cm} Non-Star-Shaped Surfaces}

\vspace{10 mm}

\textnormal{\large Brian Harvie}

\vspace{.5 in}
\begin{abstract}
{\small

We derive an upper bound on the waiting time for a variational weak solution to Inverse Mean Curvature Flow in $\mathbb{R}^{n+1}$ to become star-shaped. As a consequence, we demonstrate that any connected surface moving by the flow which is not initially a topological sphere develops a singularity or self-intersection within a prescribed time interval depending only on initial data. Finally, we establish the existence of either finite-time singularities or intersections for certain topological spheres under IMCF.}

\end{abstract}

\end{centering}

\begin{normalsize}

\section{Introduction}

Inverse mean curvature flow (IMCF) has proven to be an important tool in modern geometric analysis. Given a closed oriented manifold $N^{n}$, we say that a smooth one-parameter family of immersions $F_{t}: N^{n} \times [0,T) \rightarrow \mathbb{R}^{n+1}$ is a \textit{classical solution of inverse mean curvature flow} if

\begin{equation} \label{IMCF}
    \frac{\partial}{\partial t} F_{t}(p) = \frac{1}{H} \nu(p,t), \hspace{0.5cm} p \in N^{n}, \hspace{0.5cm} 0 \leq t \leq T.
\end{equation}%

\noindent where $H(p,t) >0$ and $\nu(p,t)$ are the mean curvature and outward unit normal of the surface $N_{t} = F_{t}(N)$ at the point $F_{t}(p)$. Although one may consider a solution of \eqref{IMCF} for hypersurfaces with boundary, c.f. \cite{Marquardt2012TheIM}, we will restrict ourselves to considering a closed manifold $N^{n}$ throughout this note.

Purely geometric applications of this flow include a proof of the Minkowski Inequality $\int_{N} H d\mu \geq 4 \pi |N|$ for outward-minimizing hypersurfaces $N \subset \mathbb{R}^{n+1}$ in \cite{huisken3} and a proof of the Poincare conjecture for manifolds with Yamabe Invariant greater than that of $\mathbb{RP}^{3}$ in \cite{Bray2004ClassificationOP}. Additionally, IMCF has been used in recent decades to solve mathematical problems in general relativity: the highest-profile of these applications has been the use of weak solutions of \eqref{IMCF} developed in \cite{huisken} to prove the Riemannian Penrose Inequality, and these solutions have subsequently been used to derive a number of other geometric inequalities in relativity (\cite{brendle}, \cite{mccormick}, \cite{wei} ).

One of the most natural questions about classical IMCF is: what conditions can one impose on a mean-convex surface $N_{0}$ to guarantee global existence? Gerhardt originally answered this question in \cite{gerhardt}, where he showed that solutions to \eqref{IMCF} are smooth, exist for all time, and homothetically approach round spheres as $t \rightarrow \infty$ if $N_{0}$ is star-shaped. Additional existence and regularity results both in Euclidean space and more exotic Riemannian manifolds typically involve obtaining first-order estimates on the support function $\omega= \langle \nu, \partial_{r} \rangle$ before obtaining second-order estimates on the second fundamental form $A$. Such approaches require $\omega$ be initially non-negative, i.e. that $N_{0}$ be star-shaped, in order to apply the appropriate maximum principles. For this reason, most literature on classical IMCF requires some star-shapedness assumption on $N_{0}$.

It is known that in general solutions to \eqref{IMCF} do not exist for all time in the non-star-shaped case. For example, a thin torus in $\mathbb{R}^{n+1}$ moving by IMCF will fatten up until the mean curvature over the inner ring reaches zero, thereby terminating the flow in finite time. With this in mind, one asks if finite-time singularities may also happen for topological spheres. As a first step toward answering these questions, we prove several results in this paper related to global existence as well as the formation of singularities and self-intersections for IMCF.

In Section 2, we demonstrate that the variational weak solutions to the flow first introduced in \cite{huisken} respect a reflection property first proposed by Chow and Gulliver in \cite{chow}. We use this property to conclude that these solutions must be star-shaped by the time they lie entirely outside of the smallest sphere they are initially enclosed by. This implies an upper bound on the ``waiting time" for a variational solution to become star-shaped depending only on the inradius and outradius of the initial surface.

Section 3 concerns the applications of this waiting time result to classical solutions. We show that, assuming initial connectedness, a classical solution defines a weak solution in the sense of \cite{huisken} if and only if it remains embedded. Using this result, we then establish a correspondence between the global variational solution of IMCF and the one defined by the classical solution. This correspondence reveals that embedded solutions to \eqref{IMCF} which exist for twice the afforementioned waiting time exist globally and homothetically converge to spheres. It also implies that all initial surfaces which are not topological spheres must either cease to be embedded or develop a finite-time singularity within twice the waiting time.

Since all initial surfaces without spherical topology must develop intersections or singularities, one naturally asks if the converse is true. That is, does the solution to \eqref{IMCF} always exist globally and remain embedded if $N_{0}$ has spherical topology? In Section 4, we show this not to be the case by constructing a mean convex with spherical topology $\mathbb{S}^{n}$ which is not outward minimizing. In particular, the corresponding weak solution must ``jump" at the initial time before the classical solution either terminates or self intersects. All results hold in any dimension. \\

\section*{Acknowledgements}
I would like to thank my thesis advisor Adam Jacob, whose weekly discussions with me made the completion of this project possible, the University of California, Davis Department of Mathematics for their financial support throughout my graduate studies, and the referee for a number of helpful comments.

\section{An Aleksandrov Reflection Approach to Level Set Solutions}

In the following section, we will demonstrate that solutions to classical IMCF respect a reflection property first explored in \cite{chow}, where Chow and Gulliver show an identical property for viscosity solutions of flows with normal speed non-decreasing in each principal curvature. Our approach applies this ``moving plane'' approach for the type of weak solution of the flow first detailed in \cite{huisken}. We begin by discussing the nature of these solutions. 

Suppose a solution $\{N_{t}\}_{0 \leq t < T}$ to \eqref{IMCF} foliates its image, that is, $N_{t_{1}} \cap N_{t_{2}} = \varnothing$ for $t_{1} \neq t_{2}$. Then it is possible to define a function $u: U= \cup_{t \in [0,T)} N_{t} \subset \mathbb{R}^{n+1} \rightarrow \mathbb{R}$ over the foliated region by $u(x)=t$ for $x \in N_{t}$ (Note this is not well defined if $x \in N_{t_{1}} \cap N_{t_{2}}$ for $t_{1} \neq t_{2}$). One can then verify that this $u$ solves the following degenerate elliptic Dirichlet problem:

\begin{eqnarray}\label{weak}
\text{div}(\frac{\nabla u}{|\nabla u|}) &=& |\nabla u| \text{       in       } U, \\
u|_{N_{0}} &=& 0. \nonumber
\end{eqnarray}

 With the level set function $u$ in mind, Huisken and Ilmanen in \cite{huisken} developed a notion of variational solutions to \eqref{weak}, and the comparison principle we shall utilize in this section crucially applies to these. 
 
 \begin{definition} \thlabel{variational}
 Given an open set $U \subset \mathbb{R}^{n+1}$, a function $u \in C^{0,1}_{\text{loc}} (U)$ is a \textbf{variational solution to IMCF} if for any $K \subset \subset U$ and $v \in C^{0,1}(K)$ with $\{ v \neq u \} \subset \subset K$ we have
 
 \begin{equation*} 
 J_{K}(u,u) \leq J_{K}(u,v)
 \end{equation*}
 
 where $J_{K}$ is the functional defined by
 
 \begin{equation} \label{min}
     J_{K}(u,v)= \int_{K} |\nabla u| + u|\nabla v|.
 \end{equation}
 
 Furthermore, given an open, bounded subset $E_{0} \subset \mathbb{R}^{n+1}$, a function $u: \mathbb{R}^{n+1} \rightarrow \mathbb{R}$ is a \textbf{variational solution to IMCF with initial condition $\mathbf{E_{0}}$} if $E_{0}=\{ u < 0\}$ and $u$ minimizes \eqref{min} on $U=\mathbb{R}^{n+1} \setminus E_{0}$.
 \end{definition}
 
 By picking the appropriate one-parameter family of test-functions, one can verify that a $C^{2}$ function $u: U \rightarrow \mathbb{R}$ with nonvanishing gradient which minimizes $J_{K}$ must satisfy $|\nabla u (x)|=H(x)$, where $H(x)$ is the mean curvature of the level set of $u$ at $x$. Thus every solution to \eqref{weak} over some open set $U$ minimizes \eqref{min} over $U$. For general variational solutions, however, there may exist points where $\nabla u =0$, and the presence of regions where the gradient of a solution vanishes also allows for the presence of points where it is not differentiable. 
 
 Huisken and Ilmanen nevertheless demonstrated that, given any open set $E_{0} \subset \mathbb{R}^{n+1}$ with $C^{1}$ boundary there is a unique variational solution $u$ with initial condition $E_{0}$ for which the sets $\{ u < t\}$ are precompact for each $t$ (Notice that if $\partial E_{0}= N_{0}$ then this means $u|_{N_{0}}=0$), and all of our results in this section apply specifically to these solutions. Key to our approach is a comparison principle from \cite{huisken} which applies to any variational solution $u$. More specifically, given any locally Lipchitz $u$ and $v$ which solve $\eqref{weak}$ over some open $U$, we know that if $\{ u < 0\} \subset \{ v < 0 \}$ then $\{ u < t \} \subset \{v < t \}$ for each $t \in \mathbb{R}$ on $U$, provided the level sets of $v$ are precompact in $\mathbb{R}^{n+1}$. Let us now give a few more definitions neccessary for our moving plane approach.
 
 Consider the plane $P_{\lambda, \nu} = \{ x \in \mathbb{R}^{n} | \langle x, \nu \rangle = \lambda \}$ with unit normal vector $\nu \in \mathbb{S}^{n}$ and upper and lower half-spaces $H^{+}_{\lambda, \nu}= \{ x \in \mathbb{R}^{n} | \langle x, \nu \rangle > \lambda \}$ and $H^{-}_{\lambda, \nu}= \{ x \in \mathbb{R}^{n} | \langle x, \nu \rangle < \lambda \}$ respectively. Let $N_{\lambda, \nu}: \mathbb{R}^{n+1} \rightarrow \mathbb{R}^{n+1}$ denote the reflection about $P_{\lambda, \nu }$. \\
 
 \begin{definition}
Given a subset $E \subset \mathbb{R}^{n+1}$, we say that $P_{\lambda, \nu}$ is \textbf{admissible with respect to $\mathbf{E}$} if $N_{\lambda, \nu} (E \cap H^{-}_ {\lambda, \nu }) \subset E \cap H^{+}_{\lambda, \nu}$.
\end{definition}

Our first result concerns the admissibility of the flow surfaces of IMCF. Given a plane $P_{\lambda, \nu}$ with corresponding reflection $N_{\lambda, \nu}: \mathbb{R}^{n+1} \rightarrow \mathbb{R}^{n+1}$, first note that if $u$ solves \eqref{weak}, then so does $u^{*}(x)= u \circ N_{\lambda, \nu} (x)$ since $N_{\lambda, \nu}$ is an isometry of $\mathbb{R}^{n+1}$. Let $E_{t}= \{ x \in \mathbb{R}^{n+1} | u(x) < t \}$ and $E^{*}_{t}= \{ x \in \mathbb{R}^{n+1} | u^{*}(x) < t \}$.

\begin{proposition}\thlabel{comparison}
For some bounded, open $E_{0} \subset \mathbb{R}^{n+1}$ with $C^{1}$ boundary, let $u: \mathbb{R}^{n+1} \rightarrow \mathbb{R}$ be the variational solution to IMCF with initial condition $E_{0}$ such that $\{ u < t \}$ is precompact for each $t$. If $E^{*}_{0} \cap H^{+}_{\lambda, \nu} \subset E_{0} \cap H^{+}_{\lambda, \nu}$, then $u^{*} (x) \geq u(x)$ for every $x \in H^{+}_{\lambda, \nu}$. In particular, $E^{*}_{t} \cap H^{+}_{\lambda, \nu} \subset E_{t} \cap H^{+}_{\lambda, \nu}$ for every $t>0$.
\end{proposition}

\begin{remark}
If $N_{t}$ is a classical solution to IMCF, then $E_{t}$ corresponds to the region enclosed by $N_{t}$. Then this theorem implies for classical solutions that if for a particular plane the portion of $N_{0}$ in the lower half-plane reflected into the upper half-plane lies inside the portion already within the upper half-plane, then this remains true for each $N_{t}$.
\end{remark}

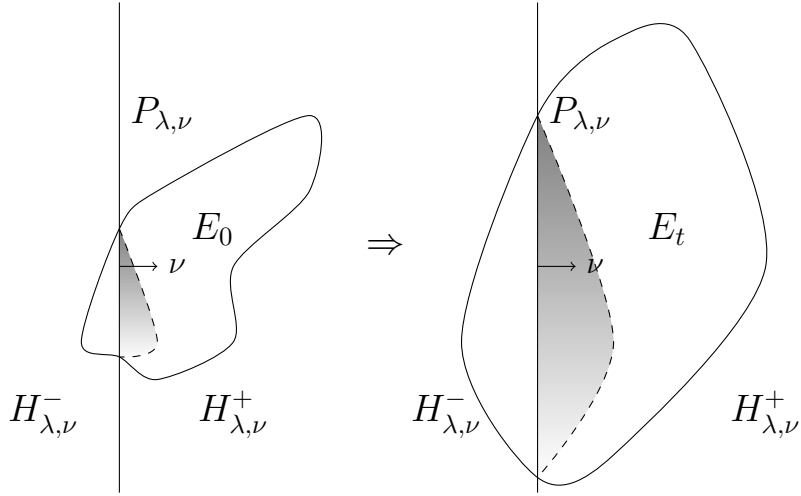
\begin{figure}
\centering
\begin{tikzpicture}
\tikzset{
    show curve controls/.style={
        decoration={
            show path construction,
            curveto code={
                \draw [blue, dashed]
                    (\tikzinputsegmentfirst) -- (\tikzinputsegmentsupporta)
                    node [at end, cross out, draw, solid, red, inner sep=2pt]{};
                \draw [blue, dashed]
                    (\tikzinputsegmentsupportb) -- (\tikzinputsegmentlast)
                    node [at start, cross out, draw, solid, red, inner sep=2pt]{};
            }
        }, decorate
    }
}

\draw [black] plot [smooth cycle] coordinates {(0,1) (-0.5,0.5) (-1,-1) (-0.5, -1.2) (0, -1.5) (1,-1) (1,0) (2,1) (2,2)};

\shadedraw [dashed] plot [smooth] coordinates {(-0.5,0.5) (0,-1) (-0.5,-1.2)};
\draw (-0.5,-3) -- (-0.5,3.5);
\draw[->] (-0.5, 0) -- (0,0);
\draw (0,0) node[anchor=west]{$\nu$};
\draw (-0.5,2) node[anchor=west] {\large{$P_{\lambda, \nu}$}};
\draw (0.3,0.5) node[anchor=west] {\large{$E_{0}$}};
\draw (1,-1.5) node[anchor=north] {\large{$H^{+}_{\lambda, \nu}$}};
\draw (-1.5,-1.5) node[anchor=north] {\large{$H^{-}_{\lambda, \nu}$}};

\draw (3,0.5) node[anchor=north] {\large{$\Rightarrow$}};
\draw [black] plot [smooth cycle] coordinates {(4,-1) (5,-2.8) (6,-2.5) (8, 0) (7,3) (6,3) (5,2)};
\shadedraw[dashed] plot [smooth] coordinates {(5,2) (6,-1) (5,-2.8)};
\draw (5, -3) -- (5, 3.5);
\draw (3.8,-1.5) node[anchor=north] {\large{$H^{-}_{\lambda, \nu}$}};
\draw (8,-1.5) node[anchor=north] {\large{$H^{+}_{\lambda, \nu}$}};
\draw (5,2) node[anchor=west] {\large{$P_{\lambda, \nu}$}};
\draw[->] (5,0) -- (5.5,0);
\draw (5.5,0) node[anchor=west]{$\nu$};
\draw (6.3,0.5) node[anchor=west] {\large{$E_{t}$}};
\end{tikzpicture}

\caption{Given a set $E_{0}$ which is initially admissible with respect to some plane, the corresponding solution $\{E_{t}\}_{0 \leq t \leq T}$ to weak IMCF remains admissible for every $t$.}
\end{figure}

\begin{proof}

From Remark 1.18 in \cite{huisken}, if $u$ is a solution of \eqref{weak} over $\mathbb{R}^{n+1}$, then so is $\min\{u,c\}$ for any constant $c \in \mathbb{R}^{+}$. Then for some $t \in \mathbb{R}^{+}$, consider the set $U= (E^{*}_{t} \setminus E_{0}) \cap H^{+}_{\lambda, \nu}$ and the cut-off solution $u^{t}=\text{min}\{ u, t \}$ to \eqref{weak}. We claim that $\{ u^{t} > u^{*} + \delta \} \subset \subset U$ for every $\delta >0$. 

Observe that $\partial U \subset ((u^{*})^{-1}\{t\} \cap H^{+}_{\lambda, \nu}) \cup (N_{0} \cap H^{+}_{\lambda, \nu}) \cup P_{\lambda, \nu}$. Since $u^{t} \leq t$, we have $u^{t} \leq u + \delta$ near $(u^{*})^{-1}\{t\} \cap H^{+}_{\lambda, \nu}$.  Since $E^{*}_{0} \cap H^{+}_{\lambda, \nu} \subset E_{0} \cap H^{+}_{\lambda, \nu}$, we have $u^{*}(x) \geq 0$ and therefore $u^{*}(x) + \delta \geq \delta \geq u(x) \geq u^{t} (x)$ near $N_{0} \cap H^{+}_{\lambda, \nu}$. Finally, since $u^{*}(x)=u(x)$ on $P_{\lambda, \nu}$, we have $u^{*}(x) + \delta \geq u(x)$ near $P_{\lambda, \nu}$. Then we may conclude that $\{ u^{t} > u^{*} + \delta \} \subset \subset U$, meaning by Theorem 2.2(i) in \cite{huisken} we get $u^{t} \leq u^{*} + \delta$ in $U$, implying $u^{t} \leq u^{*}$ in $U$. But since $u^{*}<t$ in $U$ we have $u=u^{t} \leq u^{*}$ in $U$. Since $H^{+}_{\lambda, \nu} \cap U$ is foliated by such $W$, we may conclude $u^{*}(x) \geq u(x)$ over $H^{+}_{\lambda, \nu} \cap U$. Then $E^{*}_{t} \cap H^{+}_{\lambda, \nu} \subset E_{t} \cap H^{+}_{\lambda, \nu}$, so $P_{\lambda, \nu}$ is admissible for every $E_{t}$.
\end{proof}

\begin{corollary}\thlabel{nonincreasing}
Let $E_{0}$, $u$ be as in Proposition 1. Suppose $P_{\tilde{\lambda}, \nu}$ be admissible with respect to $E_{0}$ for every $\tilde{\lambda} \in (-\infty, \lambda)$. Then $u(x)$ is nonincreasing in the $\nu$ direction over $H^{-}_{\lambda, \nu}$. 
\end{corollary}

\begin{proof}
Take $x_{1}, x_{2} \in H^{-}_{\lambda, \nu}$ which lie on the same line perpendicular to $P_{\lambda, \nu}$, i.e. $x_{1}=s_{1}\nu + y$ and $x_{2}=s_{2}\nu + y$ for $y \in P_{\lambda, \nu}$. Without loss of generality, say $s_{2} < s_{1} < 0$. 

Let $P_{\tilde{\lambda}, \nu}$ be the plane parallel to $P_{\lambda, \nu}$ which bisects $x_{1}$ and $x_{2}$. Note then that $P_{\tilde{\lambda}, \nu}$ is admissible with respect to $N_{0}$ since $\tilde{\lambda} < \lambda$. Then by \thref{comparison}, we have that $u^{*}(x)= u \circ \tilde{N}_{\lambda, \nu} (x) \geq u(x)$ for every $x \in \tilde{H}^{+}_{\lambda, \nu}$. In particular, since $x_{1} \in \tilde{H}^{+}_{\lambda, \nu}$, we must have $u^{*}(x_{1}) \geq u(x_{1})$. But $u^{*}(x_{1})=u \circ \tilde{N}_{\lambda, \nu}(x_{1})=u(x_{2})$, so $u(x_{2}) \geq u(x_{1})$. 
\end{proof}

Now we may use this result to represent the part of the surface in the lower half-plane as a locally Lipschitz graph. For the purpose of extending these results to weak solutions, we also prove this for the boundary of $E^{+}_{t}= \text{Int}(\{ x \in \mathbb{R}^{n} | u(x) \leq t \})$.

\begin{proposition}\thlabel{lipschitz}
Let $u$, $E_{0}$ be as Proposition 1. For a given $\lambda \in \mathbb{R}, \nu \in \mathbb{S}^{n}$, suppose for some $\epsilon >0$ that $P_{\tilde{\lambda}, \tilde{\nu}}$ is admissible with respect to $E_{0}$ for every $\tilde{\lambda} \in (-\infty, \lambda)$ and $\tilde{\nu}$ with $|\tilde{\nu}-\nu| < \epsilon$. Then $\partial E_{t} \cap H^{-}_{\lambda, \nu}$ and $\partial E^{+}_{t} \cap H^{-}_{\lambda, \nu}$ are each locally Lipschitz graphs in the $\nu$ direction over $P_{\lambda, \nu}$. 
\end{proposition}

\begin{proof}
We prove the result for $\partial E^{+}_{t} \cap H^{-}_{\lambda, \nu}$, as the proof for $\partial E_{t} \cap H^{-}_{\lambda, \nu}$ is identical. We begin by noting that $u$ is nonincreasing in the $\tilde{\nu}$ direction over $H_{\lambda, \tilde{\nu}}$ for every $\tilde{\nu}$ with $|\tilde{\nu} - \nu| < \epsilon$ by \thref{nonincreasing}.

Fix $x_{1}, x_{2} \in \partial E^{+}_{t} \cap H^{-}_{\lambda, \nu}$. Write $x_{1}=s_{1}\nu + y_{1}$, $x_{2}=s_{2}\nu + y_{2}$ for $y_{1}, y_{2} \in P_{\lambda, \nu_{0}}$, and say without loss of generality that $s_{1} \leq s_{2} < 0$. There exists $\tilde{\epsilon}$ so that $x_{1}, x_{2} \in H^{-}_{\lambda, \tilde{\nu}}$ for every unit vector $\tilde{\nu} \in \mathbb{S}^{n}$ satisfying $|\tilde{\nu} - \nu| < \epsilon$. Define $\hat{\epsilon}=\text{min}\{ \epsilon, \tilde{\epsilon} \}$. We will show that 
\begin{equation} \label{bound}
    |s_{1}-s_{2}| \leq \cot \hat{\epsilon} |y_{1} - y_{2}|.
\end{equation} 

To see this, suppose that \eqref{bound} is false. Then for the unit vector 

\begin{equation}
    \hat{\nu} = \frac{x_{2} - x_{1}}{|x_{2} - x_{1}|}
\end{equation}

we must have

\begin{equation}
    \langle \hat{\nu}, \nu \rangle= \frac{\langle y_{2} - y_{1}, \nu \rangle + (s_{2} - s_{1}) }{\sqrt{|y_{1} - y_{2}|^{2} + (s_{1}-s_{2})^{2}}} > \frac{1}{\sqrt{\tan^{2} \hat{\epsilon} + 1}} = \cos \hat{\epsilon}.
\end{equation}

Now, pick $\tilde{x_{2}}$ with $u(\tilde{x_{2}})> t$ sufficiently close to $x_{2}$ so that for the vector $\tilde{\nu}= \frac{\tilde{x_{2}} - x_{1}}{|\tilde{x_{2}} - x_{1}|}$ we have $\langle \tilde{\nu}, \nu \rangle > \cos(\hat{\epsilon})$ and $\langle \tilde{x}_{2}, \tilde{\nu} \rangle < \lambda$. Note that the first inequality implies $|\tilde{\nu} - \nu| < \hat{\epsilon}$. Then $P_{\lambda, \tilde{\nu}}$ is admissible with respect to $N_{0}$, and $x_{1}, \tilde{x}_{2}$ lie in $H^{-}_{\lambda, \tilde{\nu}}$. In fact, we have that $x_{1}, \tilde{x}_{2}$ lie on a line perpendicular to $P_{\lambda, \tilde{\nu}}$ with $\text{dist}\{x_{1}, P_{\lambda, \tilde{\nu}}\} > \text{dist}\{ \tilde{x}_{2}, P_{\lambda, \tilde{\nu}} \}$ by construction. But we also have that $u(x_{1})=t$ and $u(\tilde{x}_{2}) > t$, and this contradicts the nonincreasing property from \thref{nonincreasing}. 

Thus \eqref{bound} holds, and therefore $y_{1}=y_{2}$ implies $s_{1}=s_{2}$, so $\partial E^{+}_{t} \cap H^{-}_{\lambda, \nu}$, and likewise $\partial E_{t} \cap H^{-}_{\lambda, \nu}$, is a graph over $P_{\lambda, \nu}$ (Recall $\partial E_{t} = N_{t}$ for classical solutions). Furthermore, the Lipschitz bound $\cot \hat{\epsilon}$ is independent of $t$. 
\end{proof}

\begin{theorem} \thlabel{graph}
For some bounded, open $E_{0}$ with $C^{1}$ boundary, let $u: \mathbb{R}^{n+1} \rightarrow \mathbb{R}$ be the variational solution to IMCF with initial condition $E_{0}$ such that $\{ u < t \}$ is precompact for each $t$. Then, choosing $0 \in \mathbb{R}^{n+1}$ to be the midpoint of the two furthest points apart on $\partial E_{0}$, the region of the surface $\partial E_{t}$ which lies outside $B_{\frac{\text{diam}(N_{0})}{2}}(0)$ can be written as a graph $r=r_{t}(\theta)$ over $\mathbb{S}^{n}$ in polar coordinates with respect to the origin. Furthermore, this graph satisfies the gradient estimate

\begin{equation} \label{grad}
    |Dr_{t}| \leq \frac{r_{t} \Lambda}{\sqrt{r_{t}^{2} - \Lambda^{2}}} 
\end{equation}
for some $\Lambda \leq \frac{\text{diam}(N_{0})}{2}$.
\end{theorem} 

\begin{proof} We follow the proof of Theorem 4 in \cite{chow}. For a given $E_{0}$, take $0 \in \mathbb{R}^{n+1}$ to be the midpoint of the line connecting a pair of distance-maximizing points on $\partial E_{0}$. For a given $\nu \in \mathbb{S}^{n}$, define $\lambda_{\text{max}}$ to be the supremum over all $\lambda \in \mathbb{R}$ such that $P_{\tilde{\lambda}, \nu}$ is admissible with respect to $E_{0}$ for each $\tilde{\lambda} \in [-\infty, \lambda)$, then define $\Lambda= \sup_{\nu \in \mathbb{S}^{n}} -\lambda_{\text{max}}$. Then $0 \leq \Lambda \leq \frac{\text{diam}(N_{0})}{2}$. Given $x_{0} \in \partial E_{t}$ with $|x_{0}|=r_{0} > \Lambda$, we know $x_{0} \in H^{-}_{\lambda_{\text{max}}, \nu}$ for each $\nu \in \mathbb{S}^{n}$ and associated $\lambda_{\text{max}}$. Write $x_{0}= r_{0} \frac{\partial}{\partial r}$. Then for $\nu_{0}= - \frac{\partial}{\partial r}$ we have $\langle \nu_{0}, x_{0} \rangle = -r_{0} < -\Lambda$, so by \thref{lipschitz} $\partial E_{t}$ is a Lipschitz graph $r= r_{t}(\theta)$ in some neighborhood of $x_{0}$. Letting $\frac{\partial}{\partial \theta}$ be a unit tangent over $\mathbb{S}^{n}$, the vector $\tau = -r\frac{\partial}{\partial \theta} - Dr_{t} (\theta) \nu_{0}$ is tangent to $\partial E_{t}$. Also by \thref{lipschitz}, $\tau$ is transverse to $\nu$ for all $\nu \in \mathbb{S}^{n}$ with $\langle \nu, x_{0} \rangle < -\Lambda$, so

\begin{equation}
    \frac{r Dr_{t}(\theta)}{(r^{2} + (Dr_{t}(\theta))^{2})^{\frac{1}{2}}} = \langle \frac{\tau}{|\tau|}, x_{0} \rangle \geq -\Lambda.
\end{equation}

Rearranging this yields

\begin{equation}
    Dr_{t} (\theta) \leq \frac{r \Lambda}{(r^{2}- \Lambda^{2})^{\frac{1}{2}}}.
\end{equation}
\end{proof}

\begin{theorem} {(Waiting Time for Star-shapedness)} \thlabel{existence}
For bounded, open $E_{0}$ with $C^{2}$ boundary, suppose $u: \mathbb{R}^{n+1} \rightarrow \mathbb{R}$ is the variational solution to IMCF with initial condition $E_{0}$ such that the sets $\{ u < t \}$ are precompact for each $t$. Let $R$ be the inradius of $N_{0}$, that is the radius of the largest ball contained within $N_{0}$. Then the level sets $N_{t}= \partial E_{t}$ of $u$ lie entirely outside $B_{\frac{\text{diam}(N_{0})}{2}}(0)$ for any $t \geq t^{*}=  n\log{(R^{-1} \text{diam} (N_{0}))}$. In particular, $E_{t}$ is star-shaped and hence smooth for every $t \geq t_{*}$ and thus $u$ may be extended to all of $\mathbb{R}^{n+1}$.
\end{theorem}

\begin{proof}
Pick $0$ to be the midpoint between the pair of points $x,y \in N_{0}$ which maximize $|x-y|$. Then $N_{0} \subset B_{\frac{\text{diam}(N_{0})}{2}}(0)$. Since $R$ is the inradius of $N_{0}$, there exists some $x \in E_{0}$ such that $B_{R}(x) \subset E_{0}$. By Theorem 2.2 in \cite{huisken}, we must have $B_{R e^{\frac{t}{n}}}(x) \subset E_{t}$ for each $t \in [0,T)$. 
We must have that $B_{\frac{\text{diam}(N_{0})}{2}} (0) \subset B_{R e^{\frac{t_{*}}{n}}}(x)= B_{\text{diam}(N_{0})} (x)$. Conclude then that $B_{\frac{\text{diam}(N_{0})}{2}} (0) \subset E_{t}$ and thus $\partial E_{t}$ is star-shaped. \thref{existence} then follows from the long-time existence results in \cite{gerhardt}.
\end{proof} 
\begin{remark}
In Remark 2.8(b) of \cite{huisken2}, the authors suggested a similar ``waiting time" for star-shapedness of the flow depending on the diameter and area of $N_{0}$ if the reflection property was shown to apply to their variational solutions. We were unable to determine how they derived this time, and so we instead include the above one.
\end{remark}

\section{Consequences for Classical Solutions}
In this section, we show that an embedded connected classical solution of \eqref{IMCF} always gives rise to a variational weak solution. Later, we show that if this solution exists and is embedded beyond the time $2t^{*}$ defined in Theorem 2, then its flow surfaces equal the level sets of the variational solution with initial condition $E_{0}$ which has $E_{t}$ precompact. This allows us to apply Theorem 2 to these classical solutions, establishing star-shapedness beyond the time $t^{*}$. Key to showing this is a comparison principle for IMCF which is slightly weaker than the well-known two-sided avoidance principle for MCF:

\begin{theorem}{(One-Sided Avoidance Principle)}\thlabel{avoidance}
Let $N_{0} \subset \mathbb{R}^{n+1}$ be a connected, closed hypersurface, and $\{N_{t}\}_{0 \leq t < T}$ the corresponding solution to \eqref{IMCF}. Suppose $N_{t}$ is embedded for each $t \in [0,T)$, and let $E_{t} \subset \mathbb{R}^{n+1}$ be the open domain enclosed by $N_{t}$. Now let $\tilde{N}_{0} \subset E_{0}$ be a closed, connected hypersurface, and $\{\tilde{N}_{t}\}_{0 \leq t < \tilde{T}}$ the corresponding solution to \eqref{IMCF} with $\tilde{N}_{t}$ embedded for each $t \in [0,T)$. Then $\overline{\tilde{E}}_{t} \subset E_{t}$ for each $t \in [0,T)$, and $\text{dist}\{ N_{t}, \tilde{N}_{t} \}$ is non-decreasing.
\end{theorem}
\begin{proof}
Calling $\tilde{N}_{t} = \tilde{F}_{t}(\tilde{N})$, $N_{t}=F_{t}(N)$ consider the function $f: \tilde{N} \times N \times [0, T) \rightarrow \mathbb{R}$ defined by $f(p,q,t)=|\tilde{F}_{t}(p)- F_{t}(q)|^{2}$. Define $\ell: [0,T) \rightarrow \mathbb{R}$ by $\ell(t)=\min_{(p,q) \in \tilde{N} \times N} f(p,q,t)$, where $\ell(0) >0$ by hypothesis. Since $f$ is smooth and $\tilde{N} \times N$ is closed, $\ell$ is locally Lipschitz in $(0,T)$ according to Lemma 2.1.3 in \cite{carlos}. Also by this lemma, for any $t_{0} \in [0,T)$ where $\ell(t)$ is differentiable we have

\begin{equation}
    \frac{d}{dt} \ell (t_{0}) = \partial_{t} f(p_{0}, q_{0}, t_{0})
\end{equation}
for any pair of points $(p_{0},q_{0}) \in \tilde{N} \times N$ satisfying $\ell(t_{0})=f(t_{0},p,q)$. We know $\ell$ is positive at least  for small times, so let $\mathcal{A} \subset [0,T)$ be the largest interval containing $0$ over which $\ell$ is strictly positive. Note that $\overline{\tilde{E}}_{t} \subset E_{t}$ for $t \in \mathcal{A}$. Take $t_{0} \in \mathcal{A}$ where $\ell$ is differentiable and let $(p_{0},q_{0}) \in \tilde{N} \times N$ be a minimizing pair of points of $f$ at $t_{0}$. The outward pointing normals at $p_{0}$ and $q_{0}$ must be parellel, since the line segment joining $\tilde{F}_{t_{0}}(p_{0})$ and $F_{t_{0}}(q_{0})$ is contained in $\overline{E_{t}}$ and does not intersect $\tilde{E}_{t}$. Calling $\nu_{0}$ the outward unit normal at $\tilde{F}_{t_{0}}(p_{0}) \in \tilde{N}_{t_{0}}$, $\tilde{F}_{t}(q_{0}) \in N_{t_{0}}$, we consider the translated surface $N'_{t_{0}}$ defined by

\begin{equation*}
    N'_{t_{0}} = \{ x + \sqrt{\ell(t_{0})} \nu_{0} | x \in \tilde{N}_{t_{0}} \}.
\end{equation*}

$N'_{t_{0}}$ and $N_{t_{0}}$ share the same tangent plane at $\tilde{F}_{t_{0}}(p_{0}) + \sqrt{\ell(t_{0})} \nu_{0} \in N'_{t_{0}}$ and $F_{t_{0}}(q_{0}) \in N_{t_{0}}$. Since $\sqrt{\ell(t_{0})}= \text{dist}\{ N_{t_{0}}, \tilde{N}_{t_{0}} \}$, we have the inclusion $E'_{t_{0}} \subset E_{t_{0}}$, where $E'_{t_{0}}$ is the set enclosed by $N'_{t_{0}}$. Since $\tilde{F}_{t_{0}}(p_{0}) + \sqrt{\ell(t_{0})} \nu_{0}=F_{t_{0}}(q_{0})$, this inclusion particularly tells us that

\begin{equation*}
    \lambda_{i} \leq \lambda_{i}', \hspace{0.5cm} 1 \leq i \leq n
\end{equation*}
where $\lambda_{i}$ and $\lambda_{i}'$ are the principal curvatures of $N_{t_{0}}$ and $N'_{t_{0}}$ at this intersection point respectively. Translating back to $\tilde{N}_{t_{0}}$, this tells us

\begin{equation}
    H(p_{0},t_{0}) \geq H(q_{0},t_{0}).
\end{equation}

Now we compute $\partial_{t} f (p_{0}, q_{0}, t_{0})$:

\begin{eqnarray*}
    \partial_{t} f (p_{0}, q_{0}, t_{0}) &=& \partial_{t} \langle \tilde{F}_{t_{0}} (p_{0}) - F_{t_{0}} (q_{0}), \tilde{F}_{t_{0}} (p_{0}) - F_{t_{0}} (q_{0}) \rangle \\
    &=& 2 \langle \frac{\partial}{\partial t} \tilde{F}_{t_{0}} (p_{0}) - \frac{\partial}{\partial t} F_{t_{0}}(q_{0}), \tilde{F}_{t_{0}} (p_{0}) - F_{t_{0}} (q_{0}) \rangle \\
    &=& 2 \langle (\frac{1}{H(p_{0},t_{0})} - \frac{1}{H(q_{0}, t_{0})}) \nu_{0}, -\sqrt{\ell(t_{0})} \nu_{0} \rangle \\
    &=& 2 \sqrt{\ell(t_{0})} (\frac{1}{H(q_{0},t_{0})} -\frac{1}{H(p_{0},t_{0})}) \geq 0.
\end{eqnarray*}

So $\frac{d}{dt} \ell(t) \geq 0$ wherever differentiable in $\mathcal{A}$. Taking times $t_{1} < t_{2}$ in $\mathcal{A}$ and using the fact that $\ell$ has total bounded variation in $[t_{1},t_{2}]$, an application of the Fundamental Theorem of calculus reveals

\begin{equation} \label{ell}
    \ell(t_{2})= \ell(t_{1}) + \int_{t_{1}}^{t_{2}} \frac{d}{dt} \ell (t) dt \geq \ell(t_{1}).
\end{equation}
Then if $\tilde{t}=\sup \mathcal{A} < T$, we would obtain the bound $\ell(\tilde{t}) \geq \ell(0) >0$, which would contradict $\mathcal{A}$ being the largest interval containing $0$ over which $\ell$ is positive. Thus $\mathcal{A}=[0,T)$ and hence $\overline{\tilde{E}_{t}} \subset E_{t}$ over $[0,T)$. The non-decreasing property also follows from \eqref{ell}.
\end{proof}

Notice that the above argument would not work if the normal vectors at the distance-minimizing point were anti-parallel, which happens in the case that the two disjoint surfaces enclose disjoint subsets. For this same reason, initially embedded solutions to \eqref{IMCF} need not remain embedded as long as they exist. For example, two initially disjoint spheres, which eventually intersect under IMCF, respect neither a two-sided avoidance principle nor an embeddedness principle. Furthermore, the flow surfaces in this case do not foliate their image. This particularly means that, after a sufficiently long time, the two spheres will not define a weak solution to the flow, even though their classical solution continues. An application of the previous theorem shows, however, that the latter inconvenience cannot happen if the flow surfaces remain embedded.
\begin{theorem} \thlabel{embed}
Let $\{N_{t}\}_{t \in [0,T)}$ solve \eqref{IMCF} with $N_{0}$ a connected hypersurface. Then the function $u: U= \cup_{0 \leq t < T} N_{t} \subset \mathbb{R}^{n+1} \rightarrow \mathbb{R}$ given by $u(x)=t$ if $x \in N_{t}$ is well-defined and differentiable with nonvanishing gradient if and only if the corresponding $F_{t}$ are embeddings for every $t \in [0,T)$.
\end{theorem}
\begin{proof}
$\Rightarrow$ We have by hypothesis that the function $u$ over the region $U$ given by $u(x)=t$ if $x \in N_{t}$ has nonvanishing gradient. Then the flow surfaces $N_{t}$ are each level sets of $u$. Since $N_{t}$ are the compact level sets of a function with nonvanishing gradient, they are necessarily diffeomorphic to one another, and hence remain embedded. 
$\Leftarrow$ Since each $N_{t}$ is a closed, connected, embedded hypersurface, we let $E_{t}$ be defined as in Theorem 3. In order for the function given by $u(x)=t$ for $x \in N_{t}$ over the region $U$ to be well-defined, we must have that $N_{t_{1}} \cap N_{t_{2}} = \varnothing$ for $t_{1} \neq t_{2} \in [0,T)$. To show this, first assume $T$ is finite and define $\mathcal{A}$ to be the largest interval of $[0,T)$ containing $0$ with the property that $N_{t_{1}} \cap N_{t_{2}} = \varnothing$ for any $t_{1},t_{2} \in \mathcal{A}$. We demonstrate in fact that $\mathcal{A}= [0,T)$. Define $\tilde{t}= \sup \mathcal{A}$. We will argue that $\tilde{t}=T$ by contradiction.

 First notice for two times $t_{a} < t_{b}$ in $\mathcal{A}$, we have the inclusion $\overline{E}_{t_{a}} \subset E_{t_{b}}$. Indeed, for $0 <\delta < t_{b} -t_{a}$ small we know $\overline{E_{t_{a}}} \subset E_{t}$ for $t \in (t_{a}, t_{a} + \delta]$ by the positive outward flow speed. Then if $\overline{E}_{t_{a}} \not \subset E_{t_{b}}$, letting $t_{0}$ be the first time over $t \in (t_{a} + \delta,t_{b}]$ for which $\overline{E}_{t_{a}} \not \subset E_{t}$ we would have $\partial E_{t_{a}} \cap \partial E_{t_{0}} \neq \varnothing$. But this would contradict the fact that $t_{a},t_{0} \in \mathcal{A}$ so that $N_{t_{a}}$ and $N_{t_{0}}$ cannot intersect. Thus $\overline{E}_{t_{a}} \subset E_{t_{b}}$, which also means $N_{t_{a}} \subset E_{t_{b}}$.

We claim by contradiction that if $\tilde{t} < T$ then $\mathcal{A}$ is closed. Indeed, if $\mathcal{A}=[0,\tilde{t})$ then $[0,\tilde{t}]$ properly contains $\mathcal{A}$ (assuming $\tilde{t} \neq 0$, in which case $\mathcal{A}$ is automatically closed). Then there are two times $t_{1} < t_{2}$ in $[0, \tilde{t}]$ with $N_{t_{2}} \cap N_{t_{1}} \neq \varnothing$. We must have $t_{2}=\tilde{t}$ since otherwise $t_{1}, t_{2} \in \mathcal{A}$. On the other hand, the positive outward flow speed tells us that for some small $\delta > 0$, we have $\overline{\tilde{E}}_{t} \subset E_{\tilde{t}}$ for every $t \in [\tilde{t} - \delta, \tilde{t})$. But for $0 \leq t < \tilde{t} - \delta$ the above nesting result yields $\overline{E}_{t} \subset E_{\tilde{t}-\delta}$ and so $\overline{E}_{t} \subset E_{\tilde{t}}$ for each $t \in \mathcal{A}$. This implies $N_{\tilde{t}}$ cannot intersect any $N_{t}$ with $t \in \mathcal{A}$. So $\mathcal{A}=[0,\tilde{t}]$ for $\tilde{t} < T$. 

Now take $\delta < T - \tilde{t}$ and small enough so that $\overline{E}_{\tilde{t}} \subset E_{t}$ for each $t \in (\tilde{t}, \tilde{t} + \delta)$. Since $\mathcal{A} \subset [0, \tilde{t} + \delta)$, there are two times $t_{1} < t_{2}$ in $[0,\tilde{t} + \delta)$  with $N_{t_{1}} \cap N_{t_{2}} \neq \varnothing$. We cannot have $t_{1}, t_{2} \in \mathcal{A}$ by definition, and if $t_{1} \in \mathcal{A}, t_{2} \not \in \mathcal{A}$, we would get $\overline{E}_{t_{1}} \subset \overline{E}_{\tilde{t}} \subset E_{t_{2}}$ by nesting in $\mathcal{A}$, meaning $N_{t_{1}} \cap N_{t_{2}} = \varnothing$. So $\tilde{t} < t_{1} < t_{2} < \tilde{t} + \delta$. 

Define a new solution $\{\tilde{N_{t}}\}_{\tilde{t} \leq t < T - (t_{2} - t_{1})}$ to \eqref{IMCF} by $\tilde{N}_{t}=N_{t + (t_{2}-t_{1})}$. Then $N_{\tilde{t}} \subset \tilde{E}_{\tilde{t}}=E_{\tilde{t} + (t_{2} -t_{1})}$ since $0 < t_{2} - t_{1} < \delta$. By the One-Sided Avoidance Principle, this implies $\overline{E}_{t_{1}} \subset \tilde{E}_{t_{1}}=E_{t_{2}}$, but this once again contradicts $N_{t_{1}} \cap N_{t_{2}} \neq \varnothing$. Conclude $\mathcal{A}=[0,T)$. According to Lemma 2.3 in \cite{huisken}, the corresponding $u$ must then minimize \eqref{min} over $U$, and since the level sets are smooth hypersurfaces, $u$ must be differentiable with $H=|\nabla u| > 0$. The case $T= \infty$ follows via a continuation argument. 
\end{proof}

We would like to establish that if a classical solution $N_{t}$ to IMCF induces a variational solution $u$ over every $t \in [0,T)$ for sufficiently large $T$, then $N_{t}$ must be star-shaped by some time within $[0,T)$. We know this must be true for the flow surfaces of variational solution $\tilde{u}: \mathbb{R}^{n+1} \rightarrow \mathbb{R}$ with initial condition $E_{0}$ from \thref{existence}, so we seek to establish a correspondence between $u$ and $\tilde{u}$. Recall the sets $\tilde{E}_{t} = \{ \tilde{u} < t \}$ and $\tilde{E}^{+}_{t}= \text{Int}(\{ \tilde{u} \leq t \})$ from Section 2. First we observe that if $\tilde{E}_{t_{1}}$ fails to be strictly outward minimizing for some $t_{1} \in [0,T)$ (See Definition 3 in the following section), or equivalently that $\tilde{E}^{+}_{t_{1}} \neq \tilde{E}^{+}_{t_{1}}$, then the classical solution $N_{t}$ cannot fully escape the minimizing hull $\tilde{E}^{+}_{t_{1}}$ of $\tilde{E}_{t_{1}}$ before the time $T$ without self-intersecting.
\begin{lemma}[No Escape Lemma] \thlabel{escape}
Let $\{ N_{t} \}_{t \in [0,T)}$ be a solution to \eqref{IMCF} with $N_{t}$ a connected, embedded hypersurface for each $t \in [0,T)$, and $E_{t}$ as in Theorem 3. Let $\tilde{u}: \mathbb{R}^{n+1} \rightarrow \mathbb{R}$ be the variational solution to IMCF with initial condition $E_{0}$ and precompact $\tilde{E}_{t}$. Suppose there exists a time $t_{1} \in [0,T)$ so that $\tilde{E}_{t_{1}} \neq \tilde{E}^{+}_{t_{1}}$. Then there does not exist a time $t_{2} > t_{1}$ in $[0,T)$ so that $\overline{\tilde{E}^{+}_{t_{1}}} \subset E_{t_{2}}$. 
\end{lemma}

\begin{proof}
We proceed by contradiction. Define 
\begin{equation*}
    t_{1}=\inf \{t \geq 0 | \tilde{E}^{+}_{t} \neq \tilde{E}_{t} \}
\end{equation*}

By the Smooth Start Lemma 2.4 and Minimizing Hull Property 1.4 of $\tilde{E}_{t}^{+}$ from \cite{huisken}, we know for the classical solution $N_{t}$ that $N_{t}=\partial \tilde{E}_{t}$ for $t < t_{1}$. We claim that $\tilde{E}_{t_{1}} \neq \tilde{E}^{+}_{t_{1}}$. By (1.10) from \cite{huisken}, $\partial \tilde{E}_{t} = N_{t} \rightarrow \partial \tilde{E}_{t_{1}} = N_{t_{1}}$ in $C^{1,\beta}$ as $t \nearrow t_{1}$. If $\tilde{E}_{t_{1}} = \tilde{E}^{+}_{t_{1}}$, we would have since $H>0$ on $\partial \tilde{E}_{t_{1}} = N_{t_{1}}$ that $\partial \tilde{E}_{t}=N_{t}$ over some interval $[t_{1},t_{1} + \epsilon)$ by the Smooth Start Lemma. This would mean $\tilde{E}_{t}=\tilde{E}^{+}_{t}$ over $[t,t+\epsilon)$ since $N_{t}= \partial \tilde{E}_{t} \rightarrow N_{t_{0}}= \partial \tilde{E}^{+}_{t_{0}}$ in $C^{1,\beta}$ as $t \searrow t_{0}$ in $[t_{1},t_{1} + \epsilon)$ by the second part of (1.10). So W.L.O.G. we prove the result for $\tilde{E}^{+}_{t_{1}}$, as the $\tilde{E}^{+}_{t}$'s are nested in time.
$\tilde{E}_{t_{1}}^{+} \setminus \overline{\tilde{E}_{t_{1}}}$ is open by definition and nonempty by assumption, so it must have positive Hausdorff Measure. Furthermore, according to the Minimizing Hull Property 1.4(iv) and Exponential Growth Lemma 1.6 from \cite{huisken}, we have

\begin{equation} \label{equality}
    |N_{t_{1}}|= |\partial \tilde{E}_{t_{1}}|= |\partial \tilde{E}^{+}_{t_{1}}|=e^{\frac{t}{n}} |\partial E_{0}|.
\end{equation}

If there exists a $t_{2} \in [0,T)$ with $\overline{\tilde{E}^{+}_{t_{1}}} \subset E_{t_{2}}$, then take the domain $U=E_{t_{2}} \setminus E_{0}$. According to \thref{embed}, the classical solution $\{N_{t}\}_{t \in [0,t_{2})}$ induces a variational solution $u$ with nonvanishing gradient over $U$. If $\partial \tilde{E}^{+}_{t_{1}} \subset U$ we would have, in view of the positivity of $|\nabla u|$, positivity of $|\tilde{E}_{t_{1}}^{+} \setminus \tilde{E}_{t_{1}}|$, and the Divergence Theorem that

\begin{eqnarray*}
    0 &<& \int_{\tilde{E}^{+}_{t_{1}} \setminus \tilde{E}_{t_{1}}} |\nabla u| = \int_{\tilde{E}^{+}_{t_{1}} \setminus \tilde{E}_{t_{1}}} \text{div}(\frac{\nabla u}{|\nabla u|}) \\  &=& \int_{\partial \tilde{E}^{+}_{t_{1}}} \frac{\nabla u}{|\nabla u|} \cdot \nu + \int_{N_{t_{1}}} \frac{\nabla u}{|\nabla u|} \cdot \nu \\ &\leq& |\partial \tilde{E}^{+}_{t_{1}}| - |N_{t_{1}}|, 
\end{eqnarray*}

but this contradicts the equality \eqref{equality}. Conclude then that we must have $\overline{\tilde{E}^{+}_{t_{1}}} \not \subset E_{t}$ for any $t \in [0,T)$.
\end{proof}

\begin{remark}
This paper's author originally found this result for weak IMCF in an earlier version of \cite{mattia}, where it was shown instead using p-harmonic potentials. However, their proof of this theorem appears to have since been removed from \cite{mattia} for the sake of brevity.
\end{remark}

Next we confine the minimizing hull of some $\tilde{E}_{t}$ which is not strictly outward minimizing to a ball in $\mathbb{R}^{n+1}$ depending only on initial data.
\begin{lemma} \thlabel{ball}
Let $E_{0} \subset \mathbb{R}^{n+1}$ be an open bounded domain with $C^{2}$ boundary $N_{0}$, and let $\tilde{u}: \mathbb{R}^{n+1} \rightarrow \mathbb{R}$ be the variational solution with initial condition $E_{0}$ and precompact $\tilde{E}_{t}$. Choose $0 \in \mathbb{R}^{n+1}$ so that $E_{0} \subset B_{\frac{\text{diam}(N_{0})}{2}}(0)$. Then for each $t\geq 0$, we have $\tilde{E}^{+}_{t} \subset B_{e^{\frac{t}{n}}\frac{\text{diam}(N_{0})}{2}}(0)$. In particular, if $E^{+}_{t_{1}} \neq E_{t_{1}}$ for some $t_{1} \in \mathbb{R}$, then $E^{+}_{t_{1}} \subset B_{\frac{R^{-1}}{2}(\text{diam}(N_{0}))^{2}}(0)$, where $R$ is the inradius of $N_{0}$.
\end{lemma}

\begin{proof}
Observe that the sets $F_{t}=B_{e^{\frac{t}{n}}\frac{\text{diam}(N_{0})}{2}}(0)$ define a variational solution of IMCF with compact level sets and $E_{0} \subset F_{0}$, so $\tilde{E}_{t} \subset F_{t}$ by Theorem 2.2(ii) of \cite{huisken}. In the case that $\tilde{E}_{t_{1}} \neq \tilde{E}^{+}_{t_{1}}$ for some $t_{1} \geq 0$, we show $\tilde{E}^{+}_{t_{1}}$ remains contained in $F_{t_{1}}$. By definition $\tilde{E}^{+}_{t_{1}} \subset \tilde{E}_{t}$ for $t > t_{1}$. Then choosing the sequence 
\begin{equation*}
    \{t_{i}=t_{1} + n \ln{(1+i^{-1})}\}_{i=1}^{\infty},
\end{equation*}
 we have the inclusion $\tilde{E}^{+}_{t_{1}} \subset F_{t_{i}} =B_{(1 + i^{-1}){e^{\frac{t_{1}}{n}}}\frac{\text{diam}(N_{0})}{2}}(0)$. Thus 
 
 \begin{equation*}
     \tilde{E}^{+}_{t_{1}} \subset \text{Int}(\cap_{i=1}^{\infty} F_{t_{i}}) = B_{e^{\frac{t_{1}}{n}}\frac{\text{diam}(N_{0})}{2}}(0)
 \end{equation*}

For the second part of the statement, according to \thref{existence}, $\partial \tilde{E}_{t}$ is star-shaped whenever $t \geq t^{*}$. Thus $\tilde{u}$ is $C^{1}$ with $|\nabla \tilde{u}| \neq 0$ over $\mathbb{R}^{n+1} \setminus \tilde{E}_{t^{*}}$ by Theorem 0 for star-shaped hypersurfaces in \cite{gerhardt} and uniqueness. Therefore, we cannot have $\tilde{u}=t_{0}$ over a positive measure set for $t_{0} \geq t^{*}$, so $\tilde{E}_{t}=\tilde{E}^{+}_{t}$ for these times. So if $\tilde{E}^{+}_{t_{1}} \neq \tilde{E}_{t_{1}}$ then $t_{1} < t^{*}$, meaning $\tilde{E}^{+}_{t_{1}} \subset F_{t^{*}} = B_{R^{-1}(\text{diam}(N_{0}))^{2}}(0)$.

\end{proof}
Combining Lemmas 1 and 2 reveals that if the classical solution $N_{t}$ escapes the ball $B_{\frac{R^{-1}}{2}(\text{diam}(N_{0}))^{2}}(0)$ while remaining embedded, then we must have $\tilde{E}_{t}=\tilde{E}^{+}_{t}$ inside this ball. This is sufficient to ensure $N_{t}= \partial \tilde{E}_{t}$, making $N_{t}$ star-shaped beyond the time $t^{*}$. For the main theorem of this section, we estimate the time this escape takes to occur. This theorem both establishes global existence, embeddedness, and rapid convergence to spheres for $N_{t}$ existing and remaining embedded for a time greater than $2t^{*}$, and establishes the formation of singularities and self-intersections within the time $2t^{*}$ for $N_{0}$ without spherical topology. The latter is akin to the well-known upper bound on extinction time for closed surfaces moving by MCF.

\begin{theorem} {(Singularity Formation and Self-Intersection for IMCF)} \thlabel{alternative}
Let $\{N_{t}\}_{t \in [0,T_{\text{max}}]}$ be a solution to \eqref{IMCF}, where $N_{0}$ is a connected hypersurface and $T_{\text{max}}$ is the maximal time of existence. Then one of the following alternatives holds:

\begin{enumerate}
    \item $T_{\text{max}}= \infty$ and $N_{t}$ is embedded for every $t \in [0,T_{\text{max}})$. Furthermore, $N_{t}$ is star-shaped for any $t \geq t^{*}= n\log{(R^{-1} \text{diam} (N_{0}))}$, where $R$ is the radius of the largest ball enclosed by $N_{0}$.
    \item $N_{t}$ develops either a singularity or a self-intersection within the time interval $[0,2t^{*}]$ for $t^{*}$ defined above.
\end{enumerate}
\end{theorem}

 \thref{alternative} implies that strictly embedded solutions of \eqref{IMCF} which develop singularities do so within a prescribed time interval. Furthermore, this result sharply characterizes the behavior for initial data without spherical topology.
\begin{corollary} \thlabel{singularity}
Suppose a connected hypersurface $N_{0}$ is not homeomorphic to $\mathbb{S}^{n}$. Then the corresponding solution $N_{t}$ to IMCF develops either a singularity or a self-intersection by the time $2t^{*}= 2n\log{(R^{-1} \text{diam} (N_{0}))}$
\end{corollary}

\begin{proof} Let $E_{0}$ be the set enclosed by $N_{0}$, $\tilde{u}: \mathbb{R}^{n+1} \rightarrow \mathbb{R}$ be the variational solution with initial condition $E_{0}$ and precompact $\tilde{E}_{t}$, and $0 \in \mathbb{R}^{n+1}$ be chosen so that $E_{0} \subset B_{\frac{\text{diam}(N_{0})}{2}}(0)$. Suppose that the classical solution $\{N_{t}\}_{t \in [0,T)}$ to \eqref{IMCF} with initial data $N_{0}$ exists and is embedded a time $T > 2t^{*}$. We claim then that the global solution $\tilde{u}$ satisfies $\tilde{E}_{t}=\tilde{E}^{+}_{t}$ for each $t \geq 0$, and we establish this by contradiction. Take a nonnegative time $t_{1}$ so that $\tilde{E}_{t_{1}} \neq \tilde{E}^{+}_{t_{1}}$. \thref{ball} states that $\tilde{E}_{t_{1}} \subset B_{\frac{R^{-1}}{2}(\text{diam}(N_{0}))^{2}}(0)$. 

On the other hand, we may take $x \in E_{0}$ so that $B_{R}(x) \subset E_{0}$. The classical solution $\{N_{t}\}_{0 \leq t < T}$ induces a variational solution $u$ over $U=E_{T} \setminus E_{0}$, with $E_{t}$ being as in \thref{avoidance}. By the Comparison Principle 2.2 of \cite{huisken}, we must have $B_{Re^{\frac{t}{n}}} (x) \subset E_{t}$. However, evaluating at $t=2t^{*}$ we get $B_{Re^{\frac{2t^{*}}{n}}} (x)=B_{R^{-1} (\text{diam}(N_{0}))^{2}}(x)$. Then we have the containment

\begin{equation*}
    \overline{\tilde{E}^{+}_{t_{1}}} \subset \overline{B_{\frac{R^{-1}}{2}(\text{diam}(N_{0}))^{2}}(0)} \subset B_{R^{-1} (\text{diam}(N_{0}))^{2}} (x) \subset E_{2t^{*}}.
\end{equation*}
This contradicts the No-Escape Lemma. Thus we know $\tilde{E}_{t} = \tilde{E}^{+}_{t}$ for $t \geq 0$. Letting $\mathcal{A} \subset [0,T)$ be the largest interval containing $0$ over which $\partial \tilde{E}_{t}=N_{t}$, we then have $\sup{A}>0$ by Lemma 2.4 in \cite{huisken}. If $\tilde{t}=\sup{A} < T$, we would have that $E_{\tilde{t}}=\tilde{E}_{\tilde{t}}=\tilde{E}^{+}_{\tilde{t}}$ by the above result. Then since $H>0$ on $\partial \tilde{E}_{\tilde{t}}=N_{\tilde{t}}$, Lemma 2.4 and Property 1.4 would once again imply $N_{t}=\partial \tilde{E}_{t}$ over some larger interval $t \in [0,\tilde{t}+\epsilon)$. Conclude then that $\tilde{t}=T$, i.e. that $\partial \tilde{E}_{t}=N_{t}$ over $[0,T)$.

$\partial \tilde{E}_{t}$ is star-shaped for $t \geq t^{*}$ by \thref{existence}, so by Theorem 0 of \cite{gerhardt} and continuation, we must have for $T_{\max}=+\infty$ and $N_{t}$ embedded for all times. The alternative is then that $N_{t}$ does not exist or remain embedded past the time $2t^{*}$. For \thref{singularity}, solutions which satisfy the first alternative are star-shaped and therefore topological spheres for any $t>2t^{*}$, but since they are also embedded for all times $[0,T)$, this implies that $N_{0}$ must also be a topological sphere. Thus any initial surface without spherical topology necessarily satisfies the second alternative.

\end{proof}
\begin{remark}
From \cite{gerhardt}, star-shaped data are known to homothetically converge to spheres, so Theorem 4 shows that the sphere is the unique blow-down limit of embedded solutions to $\eqref{IMCF}$ which exist at least for the time $2t^{*}$.
\end{remark}
\section{Intersections and Singularities for Topological Spheres}

To conclude this note, we will prove the following:

\begin{theorem} \thlabel{rev}
There exists an $H>0$ $N_{0}^{n} \subset \mathbb{R}^{n+1}$ with spherical topology which either self-intersects or develops a singularity within the time $T_{\text{max}} \leq 2 t_{*}$ under IMCF for the time $t_{*}$ given in \thref{existence}.
\end{theorem}

In particular, this establishes that \thref{singularity} is not an if-and-only-if. Before proceeding, we introduce a definition alluded to in the previous section which the reader familiar with \cite{huisken} can skip:

\begin{definition} A subset $E \subset \mathbb{R}^{n}$ is said to be \textbf{outward minimizing} if for every $F$ containing $E$ with $F \setminus E \subset \subset \mathbb{R}^{n}$ we have $|\partial F| \leq |\partial E|$.

Furthermore, $E$ is \textbf{strictly outward minimizing} if the above inequality is strict for every $F \neq E$.
\end{definition}
One can easily see using equation \eqref{weak} that the flow surfaces $N_{t}$ for an embedded, connected classical solution to IMCF which exists globally are strictly outward minimizing. Indeed, this solution induces a variational solution $u: \mathbb{R}^{n+1} \setminus E_{0} \rightarrow \mathbb{R}$ with $|\nabla u| > 0$ by \thref{embed}. Then given any open set $F$ containing $N_{t}$, we can perform the same integration as in \eqref{equality} from \thref{escape} over $N=F \setminus E_{t}$ using the Divergence Theorem:

\begin{eqnarray*}
    0 &<& \int_{N} |\nabla u| = \int_{N} \text{div}(\frac{\nabla u}{|\nabla u|}) \nonumber \\ 
    &=& \int_{\partial F} \frac{\nabla u}{|\nabla u|} \cdot \nu + \int_{N_{t}} \frac{\nabla u}{|\nabla u|} \cdot \nu \\
    &\leq& |\partial F|- |N_{t}|.
\end{eqnarray*}

\textit{Proof of \thref{rev}} Our construction utilizes the fact that by \thref{escape}, an open set $E_{0}$ with $\partial E_{0}=N_{0}$ must be strictly outward minimizing for the classical flow $N_{t}$ to exist longer than $2t^{*}$. Therefore, we need only construct an $H>0$ topological sphere which is not strictly outward minimizing to assure that its flow develops a finite-time singularity or intersection.

Consider two disjoint balls $B(p,R)$ and $B(-p,R)$ with centerpoints $p=(p_{1}, \dots 0)$ and $-p$ and identical radii $R$, and take the Hausdorff distance $d$ between the balls to be small enough so that their union is not outward minimizing. Take the minimizing hull $E'$ of $E= B(p,R) \cup B(-p,R)$. We seek first to establish some symmetry for $E$:

\begin{proposition} \label{hull} 
$E'$ is rotationally symmetric about the $x_{1}$ axis, and $\partial E' \setminus \partial E$ is a $C^{\infty}$ minimal hypersurface.
\end{proposition}

\begin{proof}Rotations $R: \mathbb{R}^{n+1} \rightarrow \mathbb{R}^{n+1}$ about the $x_{1}$ axis are isometries of $\mathbb{R}^{n+1}$, and thus send $E'$ to the minimizing hull of $R(E)$. Since this axis contains that centerpoints of each sphere, we know these $R$ also fix the associated balls, i.e. $R(E)=E$. Then by the uniqueness of strictly minimizing hulls, $E'$ is also the minimizing hull of $R(E)$, implying $R(E')=E'$.

The regularity of $\partial E' \setminus \partial E$ follows from Theorem 5.3(ii) of \cite{mattia} (Also mentioned on page 369 of \cite{huisken}: if the singular set $\text{Sing}(\partial E' \setminus \partial E)$ is nonempty, then its Hausdorff dimension is at least $n-1$ by rotational symmetry. But this dimension cannot exceed $n-8$. Thus $\text{Sing}(\partial E' \setminus \partial E) = \varnothing$, and as the surface is smooth outside $\text{Sing}(\partial E' \setminus \partial E)$ we obtain the regularity. Furthermore, $H=0$ on this surface by (1.15) in \cite{huisken}.

\end{proof}

Next, we are going to show that the bridge joining the two spheres does not extend past their equators. This will allow us to glue the spheres together over regions away from $\partial E' \setminus \partial E$, so that the minimizing hull of the resulting surface must still include this part.

\begin{proposition}\label{hemisphere}
The set $E' \setminus E$ is contained within the set $\{ x \in \mathbb{R}^{n+1} | |x_{1}| \leq p_{1} \}$.
\end{proposition}

\begin{proof}
We claim first that $E'$ is contained within the cylinder $C_{R}= \{ x \in \mathbb{R}^{n+1} | x_{2}^{2} + \dots + x_{n +1 }^{2} \leq R^{2}$. Suppose not: define the vector field $\hat{w}=\frac{\vec{y}}{|\vec{y}|}$, where $\vec{y}(x_{1}, \dots, x_{n})= (0, x_{2}, \dots x_{n})$ points radially away from the $x_{1}$ axis, and let $u(\vec{x})= \langle \hat{w}, \vec{x} \rangle$ be the distance from this axis or ``height" of a point $\vec{x} \in \partial E' \setminus \partial E$. Since $E$ is contained within $C_{R}$ and $a=\sup_{\partial E' \setminus \partial E} u > R$ we must have that this supremum occurs at an interior point $x_{0}$ of $\partial E' \setminus \partial E$. The $n-1$ principal curvatures corresponding to rotation all must equal $\frac{1}{a}$ at $x_{0}$, and the other principal curvature must be nonnegative since $x_{0}$ is a local maximum of the height function $u$. Thus $H(x_{0})> 0$, contradicting the minimality of this complement. Thus $\sup_{\partial E' \setminus \partial E} u \leq R$, and therefore $E'$ lies in $C_{R}$.

Now, no connected component of $E' \setminus E$ lies entirely outside $\{ x \in \mathbb{R}^{n+1} | |x_{1}| < p_{1} \}$ since a single ball is strictly outward minimizing. Thus we can have $E' \setminus E$ intersect $\{ |x_{1}| \geq p_{1} \}$ only if $(E' \setminus E) \cap \{ |x_{1}| = p_{1} \} \neq \varnothing$, but this would require $E' \setminus E \not \subset C_{R}$.




\end{proof} 

Now that we have established that the $H=0$ part of $E'$ is contained between the equators of the two spheres, we are ready to construct our example. Our surface will be of class $C^{0}$ before smoothing.

Begin by attaching a cylinder of radius $r$ for some $r<R$ and finite length about the $x_{1}$ axis to the opposite end of the sphere in the $x_{1} <0$ plane (See diagram). Then attach one end of a half torus with small radius $r$ and large radius $R^{*}$ to the end of this cylinder. Attach another cylinder extending to $x_{1}=0$ to its other end, and  reflect this surface about $\{ x_{1} = 0 \}$. 

The resulting surface must not be outward minimizing, since the original spheres were not outward minimizing and, by Proposition 4, the new surface does not touch the $H=0$ part of the original hull $E$.

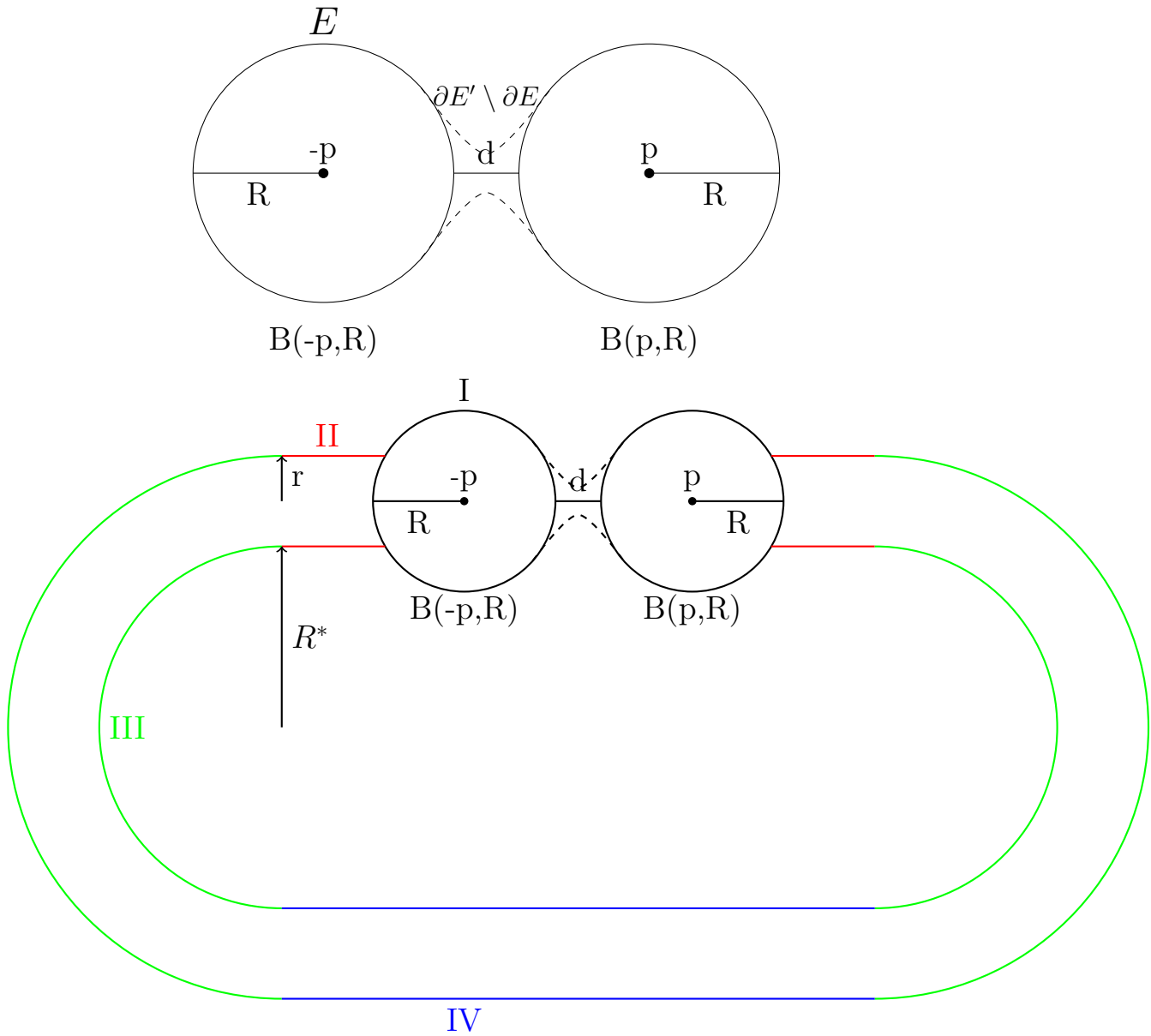
\begin{figure}

\centering
   \begin{tikzpicture}
\draw (0,0) circle(2);
\draw (0,2) node[anchor=south]{\Large{$E$}};
\draw (2.5,0.8) node[anchor=south]{$\partial E' \setminus \partial E$};
\filldraw[black] (0,0) circle(2pt) node[anchor=south] {-\large{p}};
\draw (5,0) circle(2);
\draw (-2,0) -- (0,0);
\draw (-1,0) node[anchor=north] {\large{R}};
\filldraw[black] (5,0) circle(2pt) node[anchor=south] {\large{p}};
\draw (5,0) -- (7,0);
\draw (6,0) node[anchor=north] {\large{R}};
\draw (0,-3) node[anchor=south] {\large{B(-p,R)}};
\draw (5,-3) node[anchor=south] {\large{B(p,R)}};
\draw (2,0) -- (3,0);
\draw (2.5,0) node[anchor=south] {\large{d}};

\tikzset{
    show curve controls/.style={
        decoration={
            show path construction,
            curveto code={
                \draw [blue, dashed]
                    (\tikzinputsegmentfirst) -- (\tikzinputsegmentsupporta)
                    node [at end, cross out, draw, solid, red, inner sep=2pt]{};
                \draw [blue, dashed]
                    (\tikzinputsegmentsupportb) -- (\tikzinputsegmentlast)
                    node [at start, cross out, draw, solid, red, inner sep=2pt]{};
            }
        }, decorate
    }
}

\draw [dashed] plot [smooth] coordinates {(1.5,1.3223) (2.5,0.3) (3.5,1.3222)};

\draw [dashed] plot [smooth] coordinates {(1.5,-1.3222) (2.5,-0.3) (3.5,-1.3222)};

\end{tikzpicture}

\centering
\begin{tikzpicture}[thick,scale=0.7, every node/.style={scale=1.0}]
\begin{scope}[shift={(-100,0)}]
\draw (0,0) circle(2);
\filldraw[black] (0,0) circle(2pt) node[anchor=south] {-\large{p}};
\draw (5,0) circle(2);
\draw (-2,0) -- (0,0);
\draw (-1,0) node[anchor=north] {\large{R}};
\filldraw[black] (5,0) circle(2pt) node[anchor=south] {\large{p}};
\draw (5,0) -- (7,0);
\draw (6,0) node[anchor=north] {\large{R}};
\draw (0,-3) node[anchor=south] {\large{B(-p,R)}};
\draw (5,-3) node[anchor=south] {\large{B(p,R)}};
\draw (2,0) -- (3,0);
\draw (2.5,0) node[anchor=south] {\large{d}};

\tikzset{
    show curve controls/.style={
        decoration={
            show path construction,
            curveto code={
                \draw [blue, dashed]
                    (\tikzinputsegmentfirst) -- (\tikzinputsegmentsupporta)
                    node [at end, cross out, draw, solid, red, inner sep=2pt]{};
                \draw [blue, dashed]
                    (\tikzinputsegmentsupportb) -- (\tikzinputsegmentlast)
                    node [at start, cross out, draw, solid, red, inner sep=2pt]{};
            }
        }, decorate
    }
}

\draw [dashed] plot [smooth] coordinates {(1.5,1.3223) (2.5,0.3) (3.5,1.3222)};

\draw [dashed] plot [smooth] coordinates {(1.5,-1.3222) (2.5,-0.3) (3.5,-1.3222)};

\draw [red] (-4, 1) -- (-1.732, 1);

\draw [red] (-4, -1) -- (-1.732, -1);

\draw [red] (9, 1) -- (6.732, 1);

\draw[red] (9, -1) -- (6.732, -1);

\draw[green] (-4,-1) arc(90:270:4);

\draw[green] (-4,1) arc(90:270:6);

\draw[green] (9,1) arc(90:-90:6);

\draw[green] (9,-1) arc(90:-90:4);

\draw[blue] (-4,-9) -- (9, -9);

\draw[blue] (-4,-11) -- (9,-11);

\draw (0,2) node[anchor=south] {\large{I}};

\draw (-3, 1) node[anchor=south] {\large{\color{red}{II}}};

\draw (-8, -5) node[anchor=west] {\large{\color{green}{III}}};

\draw(0, -11) node [anchor=north]
{\large{\color{blue}{IV}}};

\draw [->] (-4,-5) -- (-4, -1);

\draw(-4,-3) node[anchor=west] {\large{$ R^{*} $}};
\draw [->] (-4,0) -- (-4,1);

\draw (-4, 0.5) node[anchor=west] {\large{r}};
\end{scope}
\end{tikzpicture}
\caption{Given sufficiently close disjoint balls, one can attach handles and glue them together so that the resulting $C^{2}, H>0$ surface is not strictly outward minimizing.}
\end{figure}

It remains to show that one may refine this surface to an $H>0$ surface which is of class $C^{2}$. We require one additional lemma for this purpose.

\begin{lemma}\label{glue}
Let $U$ be any open subset of $\mathbb{R}$ containing $0$. Let $f: U \rightarrow \mathbb{R}$ be any function of the form

\begin{equation}f(x)= \begin{cases} 0 & x \leq 0 \\
g(x) & x>0
\end{cases}
\end{equation}
for some $g: U \cap \{ x>0 \} \rightarrow \mathbb{R}$. Then for every $0<\epsilon < \text{dist} \{0, \partial U\}$ there exists a function $p: (0, \epsilon) \rightarrow \mathbb{R}$ so that the the function

\begin{equation}\tilde{f}(x)= \begin{cases} 0 & x \leq 0 \\
p(x) & 0< x < \epsilon \\
g(x) & x \geq \epsilon
\end{cases}
\end{equation}

is in $C^{2}(U)$.
\end{lemma}

\begin{proof}
Take the polynomial $p(x)=Ax^{3} + Bx^{4} + Cx^{5}$ for some constants $A,B,C \in \mathbb{R}$. Clearly $p(0)=p'(0)=p''(0)=0$. Furthermore, derivatives of $p$ are related to the coefficients $A$, $B$, and $C$ by 

\begin{equation}
    \begin{pmatrix} p(x) \\ p'(x) \\ p''(x)
    \end{pmatrix} = \begin{pmatrix} x^{3} & x^{4} & x^{5} \\ 3x^{2} & 4x^{3} & 5x^{4} \\ 6x & 12x^{2} & 20x^{3} \end{pmatrix} \begin{pmatrix} A \\ B \\ C \end{pmatrix}.
\end{equation}

One may readily compute for the above matrix $M$ that $\det{M}=2x^{9} \neq 0$ for any $x \neq 0$. This means that for any triple $(X,Y,Z) \in \mathbb{R}^{3}$ and any fixed point $x \neq 0$ we may select coefficients $A$, $B$, and $C$ so that $(p(x), p'(x), p''(x))=(X,Y,Z)$. In fact, inverting the above matrix reveals that for a given $(g(\epsilon), g'(\epsilon), g''(\epsilon)) \in \mathbb{R}^{3}$

\begin{eqnarray}
    A &=& \frac{10}{\epsilon^{3}} g(\epsilon) - \frac{4}{\epsilon^{2}} g'(\epsilon) + \frac{1}{2\epsilon} g''(\epsilon) \nonumber \\
    B &=& -\frac{15}{\epsilon^{4}} g(\epsilon) + \frac{7}{\epsilon^{3}} g'(\epsilon) - \frac{1}{\epsilon^{2}} g''(\epsilon) \\
    C &=& \frac{6}{\epsilon^{5}} g(\epsilon) - \frac{3}{\epsilon^{4}} g'(\epsilon) + \frac{1}{2\epsilon^{3}} g''(\epsilon). \nonumber
\end{eqnarray}

Then restricting the domain of this $p$ to $(0,\epsilon)$, the first two derivatives of the function

\begin{equation}\tilde{f}(x)= \begin{cases} 0 & x \leq 0 \\
p(x) & 0< x < \epsilon \\
g(x) & x \geq \epsilon
\end{cases}
\end{equation}

are everywhere continuous.
\end{proof}

Now, we must establish $C^{2}$ regularity at the overlap between regions I and II, II and III, and III and IV (See figure).

\textit{Regions I-II} The union of these regions is a surface of revolution about the $x_{1}$ axis and is therefore given by a graph in the $x_{1}$ coordinate. Choose $0$ to be the point on the $x_{1}$ axis corresponding to the equator of the sphere, and let $\epsilon= \sqrt{R^{2} - r^{2}}$. Then this graph is explicitly $g(x)= f(x) + r$, where

\begin{equation}
    f(x)= \begin{cases} 0 & x \leq 0 \\
    \sqrt{R^{2} - x^{2}} - r & x > 0
    \end{cases}.
\end{equation}

Now apply Lemma 1 for this $f$ and $\epsilon$. The resulting function is $C^{2}$, and it remains only to show that the corresponding surface of revolution will be mean convex if the tube radius $r$ is sufficiently close to the sphere radius $R$. Explicitly computing the interpolating polynomial $p(x)$ by inverting the matrix in the proof of Lemma 1, we find

\begin{eqnarray}
A &=& \frac{4}{\epsilon} \frac{1}{r} - \frac{1}{2\epsilon} \frac{R^{2}}{r^{3}} \leq \frac{7}{2\epsilon} \frac{1}{r}  \nonumber \\
B &=& \frac{-7}{\epsilon^{2}} \frac{1}{r} + \frac{1}{\epsilon^{2}} \frac{R^{2}}{r^{3}} = \frac{1}{r^{3}} - \frac{6}{\epsilon^{2} r} \\
C &=& \frac{3}{\epsilon^{3}} \frac{1}{r} - \frac{1}{2\epsilon^{3}} \frac{R^{2}}{r^{3}} \leq \frac{5}{2 \epsilon^{3}} \frac{1}{r}. \nonumber
\end{eqnarray}

In particular, for every $x \in (0,\epsilon)$ $p(x)=r + Ax^{3} + Bx^{4} + Cx^{5}$ and $p''(x)=6Ax + 12Bx^{2} + 20Cx^{3}$ obey the estimates

\begin{eqnarray}
   p(x) & \leq & r + A \epsilon^{3} + B \epsilon^{4} + C \epsilon^{5} = r + \frac{\epsilon^{2}}{r^{3}} \\
   p''(x) & \leq & 6A \epsilon + 12B \epsilon^{2} + 20C \epsilon^{3} \leq \frac{12}{r^{3}} \epsilon^{2}.
\end{eqnarray}

Then choosing $\epsilon$ small enough to ensure the $H=0$ part of the original minimizing hull strictly lies in the region $\{ x_{1} > \epsilon \}$ and that $p(x)p''(x) < 1$, the $C^{2}$ surface of revolution is not outward minimizing and has $H= \frac{1}{(1+f'(x)^{2})^{\frac{3}{2}}} (1 + f'(x)^{2} - f(x) f''(x)) >0$. \\

\textit{Regions II-III/III-IV:} One may apply an identical gluing construction to each of these overlap regions, so we only present the construction for Regions III-IV here. The union of regions III and IV corresponds to a curve which is the union of a semicircle of a line. Parametrizing the lower half of the semicircle and the line as 

\begin{equation}
    g(x_{1})= \begin{cases} 0 & x<0 \\ -\sqrt{(R^{*})^{2} - x^{2}} + R^{*} & x \geq 0 \end{cases}.
\end{equation}

Here we chose the origin to be the point where the arc meets the line. For some sufficiently small $\epsilon>0$, we apply Lemma 1. It remains only to show that the surface obtained by taking a circle of radius $r$ in each plane normal to the curve at each point is mean convex for $R^{*}$ sufficiently large. From Lemma 1, the interpolating polynomial $p(x)$ has second derivative given by

\begin{eqnarray}
    p''(x)&=& 6(\frac{10}{\epsilon^{3}}(-\sqrt{(R^{*})^{2} - \epsilon^{2}} + R^{*}) - \frac{4}{\epsilon^{2}}(\frac{\epsilon}{((R^{*})^{2}-\epsilon^{2})^{\frac{3}{2}}}) + \frac{1}{2\epsilon} \frac{(R^{*})^{2}}{((R^{*})^{2} - \epsilon^{2})^{\frac{3}{2}}})x  \nonumber \\ & & + 12(\frac{-15}{\epsilon^{4}}(-\sqrt{(R^{*})^{2} - \epsilon^{2}} + R^{*}) + \frac{7}{\epsilon^{3}}(\frac{\epsilon}{((R^{*})^{2}-\epsilon^{2})^{\frac{3}{2}}}) + \frac{-1}{\epsilon^{2}} \frac{(R^{*})^{2}}{((R^{*})^{2} - \epsilon^{2})^{\frac{3}{2}}})x^{2} \nonumber \\
    & & + 20(\frac{6}{\epsilon^{5}}(-\sqrt{(R^{*})^{2} - \epsilon^{2}} + R^{*}) - \frac{3}{\epsilon^{4}}(\frac{\epsilon}{((R^{*})^{2}-\epsilon^{2})^{\frac{3}{2}}}) + \frac{1}{2\epsilon^{3}} \frac{(R^{*})^{2}}{((R^{*})^{2} - \epsilon^{2})^{\frac{3}{2}}})x^{3}. \nonumber
\end{eqnarray}

Since $0 < x < \epsilon$, we have

\begin{eqnarray}
p''(x) &\leq& \frac{180}{\epsilon^{2}} (-\sqrt{(R^{*})^{2}-\epsilon^{2}} + R^{*}) + 13 \frac{(R^{*})^{2}}{((R^{*})^{2}-\epsilon^{2})^{\frac{3}{2}}} +  \frac{84}{((R^{*})^{2}-\epsilon^{2})^{\frac{3}{2}}}.
\end{eqnarray}

For a fixed $\epsilon$ each of these terms can be made arbitrarily small by choosing $R^{*}$ large enough to guarantee that $p''(x) \leq \frac{1}{r}$ for every $x \in [0, \epsilon)$, where $\frac{1}{r}$ is the curvature of the surface in a direction orthogonal to the graph. This in turn guarantees that $H>0$ in this region, so the entire surface is $C^{2}$ and mean convex. 

\begin{remark}
Since the spheres in this construction can be chosen to be arbitrarily close to one another without changing the initial flow speed at the closest points, we suspect that this surface develops an intersection rather than a singularity first.
\end{remark}
\printbibliography[title={References}]

\begin{center}
\textnormal{ \large Department of Mathematics,
University of California, Davis \\
Davis, CA 95616\\
e-mail: bharvie@math.ucdavis.edu}\\
\end{center}
\end{normalsize}
\end{document}